\documentclass[12pt,reqno]{amsart}
\usepackage{amsmath,amsthm,amscd,amsfonts,amssymb,color}
\usepackage[mathscr]{eucal}

\usepackage[bookmarksnumbered,colorlinks,plainpages]{hyperref}
\newtheorem{theorem}{Theorem}[section]

\newtheorem{corollary}[theorem]{Corollary}

\theoremstyle{definition}
\newtheorem{definition}[theorem]{Definition}
\newtheorem{example}[theorem]{Example}
\newtheorem{Open Prob}[theorem]{Open Problem}
\theoremstyle{remark}
\newtheorem{remark}[theorem]{Remark}
\numberwithin{equation}{section}
\def\DJ{\leavevmode\setbox0=\hbox{D}\kern0pt\rlap{\kern.04em\raise.188\ht0\hbox{-}}D}

\begin{document}

\title[Farthest point problem]{A note on remotal and uniquely remotal sets in normed linear spaces}

\author[U.\ Mukherjee, S.\ Som]
{Uddalak Mukherjee, Sumit Som*}

\thanks{*Corresponding author: (somkakdwip@gmail.com)}

\address{           Uddalak Mukherjee,
                    Department of Computer Science, School of Mathematical Sciences, Ramakrishna Mission Vivekananda Educational and Research Institute, Howrah-711202, India and Advanced Computing and Microelectronics Unit, Indian Statistical Institute, Kolkata.}
                    \email{uddalakmukherjee49@gmail.com}


\address{          Sumit Som,
                   Department of Mathematics, Dinabandhu Mahavidyalaya, Bongaon-743235, North 24 Parganas, West Bengal  India.}
                    \email{somkakdwip@gmail.com}

\subjclass {$40A35,$ $47H10$}

\keywords{Farthest point, remotal set, uniquely remotal set, $\alpha\beta$-statistical convergence.}

\begin{abstract}
Remotal and uniquely remotal sets play an important role in the area of farthest point problem as well as nearest point problem in a Banach space $X.$ In this study, we find some sufficient conditions for remotality and uniquely remotality of a bounded subset of a Banach space $X$ through $\alpha\beta$-statistical convergence. 
\end{abstract}

\maketitle


\section{Introduction}
Let $X$ be a real normed linear space and $G$ be a non-empty, bounded subset of $X.$ For any $x \in X$, the farthest distance from $x$ to $G$ is denoted by $\delta(x,G)$ and is defined by
$$
\delta(x, G)=\sup \{\|x-e\|: e \in G\} \text {. }
$$
If the distance is attained, then the collection of all such points of $G$ corresponding to $x \in X$ is denoted by $F(x,G)$ and defined by
$$
F(x,G)=\{e \in G:\|x-e\|=\delta(x, G)\} \text {. }
$$
For a non-empty and bounded subset $G$ of $X,$ let us define
$$
r(G)=\{x \in X: F(x, G) \neq \phi\} .
$$
$G$ is said to be remotal if $r(G)=X$ and uniquely remotal if $r(G)=X$ and $F(x, G)$ is a singleton for each $x \in X$. The farthest point problem (FPP) states that whether every uniquely remotal set in a Banach space must be a singleton. In the year 1940, the farthest point problem was proposed by B. Jessen \cite{jes}. He proved that if $M$ is a closed uniquely remotal subset in a finite dimensional normed linear space $X,$ then $M$ is a singleton. In 1953, Motzkin, Starus and Valentine \cite{MO} analyzed uniquely remotal sets in the context of Minkowski plane and also studied remotal sets in Euclidean plane $E_2.$ In 1961, FPP for Banach spaces was introduced by Klee \cite{klee} and he proved that every compact uniquely remotal subset of a Banach space $X$ is a singleton. In 2001 Baronti and Papini \cite{ba} indicated characterizations of inner product spaces and infinite dimensional Banach spaces in terms of remotal sets and uniquely remotal sets. In \cite{CBM2}, Sababheh and Khalil proved that in a reflexive Banach space $X$ every closed bounded convex subset of $X$ is remotal if and only if $X$ is finite dimensional and showed by an example that every infinite dimensional reflexive Banach space X has a closed bounded and convex non-remotal subset. In the same paper \cite{CBM2}, they asked the question that in a Banach space $X$ (not necessarily reflexive) if every closed bounded convex subset is remotal then does it imply $X$ is finite dimensional or not. In the year 2010, Mart\'{i}n and Rao \cite{mrao} showed that in every infinite dimensional Banach space $X,$ there is a closed bounded convex subset which is not remotal. Narang et al. \cite{narang} proved the singletonness of remotal and uniquely remotal sets in the context of convex metric spaces. In 2017, Sababheh et al. \cite{CBM} introduced the notion of partial continuity of a function and and provided sufficient conditions for a closed bounded uniquely remotal subset (admitting a center) in a normed linear space to be a singleton. A Chebyshev center of a subset $E$ of a normed linear space $X$ is an element $c \in X$ such that $\delta(c, E)=\inf _{x \in X} \delta(x, E)$. In \cite{CBM1}, Sababheh et al. introduced the notion of $x$-compactness with respect to a closed bounded subset $E$ of a Banach space $X$ and $x\in X.$ On the other hand, the idea of convergence of a real sequence has been extended to statistical convergence by Fast \cite{fa} and Steinhaus \cite{st} and later on re-introduced by Schoenberg \cite{sc} independently and is based on the notion of asymptotic density or natural density of the subset of natural numbers. For a non-empty subset $K\subseteq \mathbb{N},$ the upper density is defined as $\overline{D(K)}=\displaystyle{\limsup_{n \rightarrow \infty}}\frac{1}{n}|\{k \leq n : k \in K\}|$ and the lower density is defined as $\underline{D(K)}=\displaystyle{\liminf_{n \rightarrow \infty}}\frac{1}{n}|\{k \leq n : k \in K\}|.$ We say that the natural density of $K$ exists if $\overline{D(K)}=\underline{D(K)}$ and the limit is defined as $D(K)=\displaystyle{\lim_{n \rightarrow \infty}}\frac{1}{n}|\{k \leq n : k \in K\}|.$ A sequence $\{x_n\}_{n \in \mathbb{N}}$ of real numbers is said to be statistically convergent to $x\in \mathbb{R}$ if for every $\varepsilon>0,$ $D(\{k \in \mathbb{N}: |x_k-x|\geq \varepsilon\})=0.$ If a real sequence $(x_n)$ is convergent to a real number $l$ then $(x_n)$ is statistically convergent to $l$ but the converse is not true. Also, a statistically convergent real sequence may not be bounded. For more details on statistical convergence and related results, readers can see the references \cite{c1,c2}. In the year 2021, Basu et al. \cite{ssm}, introduced the notion of partial statistical continuity of a function and proved some results about farthest point problem. In this paper, we mainly use the notion of $\alpha\beta$-statistical convergence, introduced by Aktuglu in 2014 \cite{PMK} which is more general than the notion of statistical convergence. We consider a class of functions from $[0,\infty)$ into $[0,\infty)$ to provide sufficient conditions for remotality of a non-empty bounded subsets of a Banach space $X.$ In the last part of this paper we introduce the notion of $x$-$\alpha\beta$-compactness with respect to a closed bounded subset $E$ of a Banach space $X$ and $x\in X.$ We show by an example that this notion of $x$-$\alpha\beta$-compactness is more weaker and general than $x$-compactness introduced by Sababheh et al. in \cite{CBM1}. We also use this notion to provide sufficient conditions for remotality and uniquely remotality of closed bounded subsets of a Banach space $X.$

\section{Main Results}

Firstly we recall the notion of $\alpha\beta$-statistical convergence from \cite{PMK} as follows.

\newtheorem{defn}{Definition}

\begin{definition}\label{y1}
    Let $\alpha=\left\{\alpha_n\right\}_{n \in \mathbb{N}}$ and $\beta=\left\{\beta_n\right\}_{n \in \mathbb{N}}$  be two sequences of positive numbers satisfying the following conditions: P1 : $\alpha$ and $\beta$ are both non-decreasing, P2 : $\beta_n \ge \alpha_n$ for all $n \in \mathbb{N},$ P3 : $(\beta_n - \alpha_n) \rightarrow \infty$ as $n \rightarrow \infty$. Then a sequence $\left\{x_n\right\}_{n \in \mathbb{N}}$ is said to be $\alpha\beta$-statistically convergent to a point $x \in \mathbb{R}$ if for each $\varepsilon > 0$,  		 

\vspace{0.3cm}

\centering
   $\lim _{n \rightarrow \infty} \frac{1}{(\beta_n-\alpha_{n}+1)} \mid\left\{k \in\left[\alpha_n, \beta_n\right] :| x_k-x| \geq \varepsilon\right\} \mid = 0.$ 
\end{definition}

This pair of sequences satisfying the three conditions is denoted by $(\alpha,\beta).$ In the above definition, if $\alpha_n=1$ and $\beta_n=n$ for all $n \in \mathbb{N}$ then $\alpha\beta$-statistical convergence coincides with the notion of statistical convergence. So, $\alpha\beta$-statistical convergence is more general than the notion of statistical convergence. Throughout the paper when we mention $\alpha=\left\{\alpha_n\right\}_{n \in \mathbb{N}}$ and $\beta=\left\{\beta_n\right\}_{n \in \mathbb{N}}$  be two sequences of positive numbers, we mean they satisfy the three conditions mentioned in definition \ref{y1}. Now, we introduce the notion of $\alpha\beta$-partial statistical continuity of a function at any point in a real normed linear space $X.$ For any subset $K\subseteq \mathbb{N},$ the $\alpha\beta$-density of $K$ is defined as
$$D_{\alpha\beta}(K)=\displaystyle{\lim_{n\rightarrow \infty}}\frac{1}{(\beta_n-\alpha_n+1)}|\{k\in [\alpha_n,\beta_n]: k\in K\}.$$

\begin{definition}\label{x1}
Let $X$ be a real normed linear space and $M \subseteq X$. A function $G: M \rightarrow X$ is said to be partially $\alpha\beta$-statistically continuous at $x \in M$ if there exists a non-eventually constant sequence $\left\{x_n\right\}_{n \in \mathbb{N}} \subset M$ such that $\left\{x_n\right\}_{n \in \mathbb{N}}$ is $\alpha \beta$-statistically convergent to $x$ implies $\left\{G\left(x_n\right)\right\}_{n \in \mathbb{N}}$ is $\alpha \beta$- statistically convergent to $G(x)$ i.e. for each $\varepsilon>0$,
$$
\lim _{\mathrm{n} \rightarrow \infty} \frac{1}{(\beta_n-\alpha_n+1)}\left|\left\{k \in\left[\alpha_n, \beta_n\right]:\left\|x_k-x\right\| \geq \varepsilon\right\}\right|=0
$$
implies, $\quad \lim _{\mathrm{n} \rightarrow \infty} \frac{1}{(\beta_n-\alpha_n+1)}\left|\left\{k \in\left[\alpha_n, \beta_n\right]:\left\|G\left(x_k\right)-G(x)\right\| \geq \varepsilon\right\}\right|=0.$
\end{definition}

\begin{remark}
Note that, Definition \ref{x1}, improves the Definition 2.8 in \cite{ssm} as $\alpha\beta$-statistical convergence is more general than the statistical convergence. Also, Definition \ref{x1}, improves the Definition 1.1 in \cite{CBM} as partial statistical continuity implies partial $\alpha\beta$-statistical continuity but the converse is not true. The following example illustrate this fact.
\end{remark}

\begin{example}
Let $f:[-1,1] \rightarrow \mathbb{R}$ be defined by 
$$
f(x)=\left\{\begin{array}{cc}
-1, & x<0; \\
0, & x=0; \\
1, & x>0.
\end{array}\right.
$$
It follows easily that the function in not partially continuous at $x=0$. But, $f$ is partially $\alpha\beta$-statistically continuous at $x=0$.
Let,
$$
x_n=\left\{\begin{array}{cc}
0, & n \neq 2^m~\mbox{for any}~m\in \mathbb{N}; \\
-1+c^n, &  n = 2^m~\mbox{for some}~m\in \mathbb{N}.
\end{array}\right.
$$
where $0<c<1$ and $\left[\alpha_n, \beta_n\right]=\left[1, n^2\right].$
Note that, the sequence $\left\{x_n\right\}_{n \in \mathbb{N}}$ is not convergent to 0 in the usual sense. Let $\varepsilon>0$. Now,
$$
\begin{aligned}
& \left\{k \in\left[\alpha_n, \beta_n\right]:\left|x_k-0\right| \geqslant \varepsilon\right\} \\
& \subseteq \left\{k \in\left[\alpha_n, \beta_n\right]: k=2^m \text { for some \hspace{0.1cm}} m \in \mathbb{N}\right\}.
\end{aligned}
$$
So,
$
\begin{aligned}
& \frac{1}{(n^2-1+1)}|\left\{k \in\left[1, n^2\right]:\left|x_k-0\right| \geqslant \varepsilon\right\}| 
& =\frac{\left\lfloor\log _2 n^2\right\rfloor}{n^2} \leqslant \frac{2 \log _2 n}{n^2} \rightarrow 0  
\hspace{0.1cm}
(as \hspace{0.1cm}n \rightarrow \infty).
\end{aligned}
$

Hence, $\lim _{n \rightarrow \infty} \frac{1}{n^2}\left|\left\{k \in\left[1, n^2\right]:\left|x_k-0\right| \geqslant \varepsilon\right\}\right|=0.$ So, the sequence is $\alpha \beta$-statistically convergent to $0$. Now,
$$
f\left(x_n\right)=\left\{\begin{array}{cc}
0, & \text { if \hspace{0.1cm} } n \neq 2^m \text { \hspace{0.1cm} for\hspace{0.1cm} all \hspace{0.1cm}} m \in \mathbb{N} \\
-1, & \text { if \hspace{0.1cm} } n=2^m \text {\hspace{0.1cm} for\hspace{0.1cm} some  \hspace{0.1cm}} m \in \mathbb{N}.
\end{array}\right.
$$
\end{example}
Proceeding similarly as $\left\{x_n\right\}_{n\in \mathbb{N}}$, it can be shown that the sequence $\left\{f\left(x_n\right)\right\}_{n\in \mathbb{N}}$ is $\alpha \beta$-statistically convergent to $f(0)=0.$ So, $f$ is partially $\alpha \beta$-statistically continuous at the point $x=0$.






\begin{theorem}
Let $E \subset X$ be uniquely remotal and has a Chebyshev center $c \in X$. If the farthest point map $F: X \rightarrow E$ restricted to $[c, F(c)]$ is partially $\alpha\beta$-statistically continuous at $c$, then $E$ is a singleton.
\end{theorem}

\begin{proof}
Proof is similar to \cite{ssm}, so omitted.    
\end{proof}

Now, we consider the following class of functions which will help us to get sufficient conditions for a bounded subset to be remotal in a Banach space $X.$
Let $M$ denotes the class of functions $\varphi:[0, \infty) \rightarrow [0, \infty)$ such that $\varphi$ is strictly increasing, continuous and $\varphi(0)=0.$ This class $M$ is non-empty since, $\varphi(x)=x^p, 1 \leq p<\infty, x\in [0,\infty)$ belongs to $M.$ We introduce the next definition which will be needed in our next result.

\begin{definition}
Let $\{x_n\}_{n\in \mathbb{N}}$ be a sequence of real numbers. We say that $\{x_n\}_{n\in \mathbb{N}}$ is $\alpha\beta$-statistically divergent to $\infty$ if for every $M>0,M \in \mathbb{R},$ the $\alpha\beta$-density of the set $\{k\in \mathbb{N}: x_k< M\}$ is zero.     
\end{definition}
Note that, if a sequence $\{x_n\}_{n\in \mathbb{N}}$ of real numbers is divergent to $\infty$ then the sequence $\{x_n\}_{n\in \mathbb{N}}$ is $\alpha\beta$-statistically divergent to $\infty$ for any pair $(\alpha,\beta).$ But the converse is not true.

\begin{example}
Define the sequence $\{x_n\}_{n\in \mathbb{N}}$ by
\[
x_n=\begin{cases}
0, & n=2^k\text{ for some }k\ge1,\\[4pt]
n, & \text{otherwise.}
\end{cases}
\]
Choose window sequences $(\alpha_n)$ and $(\beta_n)$ with $\beta_n\ge\alpha_n$ and lengths $\ell_n:=(\beta_n-\alpha_n+1)\to\infty$, as $n \to \infty$ and suppose the windows satisfy
\[
\frac{\log \beta_n}{\ell_n}\xrightarrow[n\to\infty]{}0.
\tag{*}
\]
Then for every fixed $M>0$ the set $\{n:\;x_n\le M\}$ has $(\alpha,\beta)$-upper density $0$, so $(x_n)$ is $(\alpha,\beta)$-statistically divergent to $+\infty$.

Fix $M>0$ and let
\[
T_M:=\{n:\;x_n\le M\}.
\]
Thus for each window $[\alpha_n,\beta_n],$ we have the crude bound
\[
|T_M\cap[\alpha_n,\beta_n]|
\le |\{m\in[\alpha_n,\beta_n]:\ m\le M\}|+ |\{k:\;2^k\in[\alpha_n,\beta_n]\}|.
\]
The first term on the right is bounded by $M$ (independently of $n$). The second term is at most the number of powers of two $\le\beta_n$, which is $\lfloor\log_2\beta_n\rfloor+1$. Hence
\[
|T_M\cap[\alpha_n,\beta_n]|\le M + \lfloor\log_2\beta_n\rfloor+1.
\]
Dividing by $\ell_n$ and taking $\limsup_{n\to\infty}$ yields
\[
D^*_{\alpha,\beta}(T_M)
\le \limsup_{n\to\infty}\frac{M+1+\lfloor\log_2\beta_n\rfloor}{\ell_n}
= \limsup_{n\to\infty}\frac{\log_2\beta_n}{\ell_n}.
\]
Under assumption \((*)\), the right-hand side is $0$. Therefore $D^*_{\alpha,\beta}(T_M)=0$. Since $M>0$ was arbitrary, $\{x_n\}_{n\in \mathbb{N}}$ is $(\alpha,\beta)$-statistically divergent to $+\infty$.
\end{example}

\begin{theorem}
Let $E$ be a non-empty bounded subset of a Banach space $X$. If there exists $\varphi \in M$ such that for all $x \in X$, there exists a sequence $\left\{x_n\right\}_{n \in \mathbb{N}} \subset X$ which converges to $x$ and $y \in E$ such that for all $z \in E$,
$$
\{\varphi(\left\|x_n-y\right\|)-\varphi(\left\|x_n-z\right\|)\}_{n \in \mathbb{N}}~\mbox{is}~\alpha\beta~\mbox{statistically divergent to}~\infty 
$$
then, $E$ is remotal in $X.$
\end{theorem}

\begin{proof}
Since, $\{\varphi(\left\|x_n-y\right\|)-\varphi(\left\|x_n-z\right\|)\}_{n \in \mathbb{N}}$ is $\alpha\beta$-statistically divergent to $\infty$ so, we have,
$$\displaystyle{\lim_{n \rightarrow \infty}}\frac{1}{(\beta_n-\alpha_n+1)}|\{k\in [\alpha_n,\beta_n]: \varphi(\left\|x_k-y\right\|)-\varphi(\left\|x_k-z\right\|)< \varepsilon\}|=0.$$

Now, the set $\{k\in [\alpha_n,\beta_n]: \varphi(\left\|x_k-y\right\|)-\varphi(\left\|x_k-z\right\|)\geq \varepsilon\}$ has $\alpha\beta$-density 1.

So, there exists $k \in\left[\alpha_n, \beta_n\right]$ such that,
$$
\begin{aligned}
& \varphi(\left\|x_k-y\right\|)-\varphi(\left\|x_k-z\right\|)\geq \varepsilon; \\
\Rightarrow & \varphi(\left\|x_k-y\right\|)\geq \varphi(\left\|x_k-z\right\|)+\varepsilon; \\
\Rightarrow & \varphi(\left\|x_k-y\right\|)>\varphi(\left\|x_k-z\right\|).
\end{aligned}
$$
Taking $k \rightarrow \infty$ we get,
$$
\varphi(\|x-y\|)\geq \varphi(\|x-z\|).
$$

As, $\varphi$ is strictly increasing, so,
$$
\|x-y\|\geq \|x-z\|.
$$
So, $E$ is remotal in $X.$
\end{proof}

\begin{theorem}
Let $E$ be a nonempty bounded subset of a Banach space $X.$ If there exists $\varphi \in M$, such that, for all $x \in X$,there exists a sequence $\left\{x_n\right\}_{n \in \mathbb{N}} \subset X$ which converges to $x$, and $y \in E$, such that for all $z \in E -\{y\},$ $\frac{\varphi(\left\|x_n-x\right\|)}{\varphi(\left\|x_n-y\right\|)-\varphi(\left\|x_n-z\right\|)}$ converges $\alpha \beta$-statistically to 0,
then $E$ is remotal.
\end{theorem}

\begin{proof}

Since, $\frac{\varphi(\left\|x_n-x\right\|)}{\varphi(\left\|x_n-y\right\|)-\varphi(\left\|x_n-z\right\|)}$ is $\alpha\beta$-statistically convergent to 0, so, 
$$\displaystyle{\lim_{n \rightarrow \infty}}\frac{1}{(\beta_n-\alpha_n+1)}|\{k\in [\alpha_n,\beta_n]: |\frac{\varphi(\left\|x_n-x\right\|)}{\varphi(\left\|x_n-y\right\|)-\varphi(\left\|x_n-z\right\|)}|\geq \varepsilon\}|=0.$$ As the $\alpha\beta$-density of the set $\{k\in \mathbb{N}: |\frac{\varphi(\left\|x_k-x\right\|)}{\varphi(\left\|x_k-y\right\|)-\varphi(\left\|x_k-z\right\|)}|< \varepsilon\}|$ is 1, so, we must have some large enough $k \in\left[\alpha_n, \beta_n\right]$ such that,
$$
\begin{aligned}
& \quad\left|\frac{\varphi(\left\|x_k-x\right\|)}{\varphi(\left\|x_k-y\right\|)-\varphi(\left\|x_k-z\right\|)}\right| < \varepsilon. \\
& \Rightarrow \frac{\varphi(\left\|x_k-x\right\|)}{\varphi(\left\|x_k-y\right\|)-\varphi(\left\|x_k-z\right\|)}<\varepsilon \hspace{0.5cm} \textnormal { or } \hspace{0.5cm} \frac{\varphi(\left\|x_k-x\right\|)}{\varphi(\left\|x_k-y\right\|)-\varphi(\left\|x_k-z\right\|)}>-\varepsilon. \\
\end{aligned}
$$
$$
\begin{aligned}
\textnormal { If } \frac{\varphi(\left\|x_k-x\right\|)}{\varphi(\left\|x_k-y\right\|)-\varphi(\left\|x_k-z\right\|)}<\varepsilon \Rightarrow \varphi(\left\|x_k-y\right\|)>\varphi(\left\|x_k-z\right\|)+\varepsilon_0 \varphi(\left\|x_k-x\right\|).
\end{aligned}
$$


Now, taking limit $k \rightarrow \infty$ as $\beta_n \rightarrow \infty$, we get,
$$
\varphi(\|x-y\|)\ge\varphi(\|x-z\|).
$$

Also, since $\varphi$ is strictly increasing,
$$
\|x-y\|\ge\|x-z\|.
$$

The case for $\frac{\varphi(\left\|x_k-x\right\|)}{\varphi(\left\|x_k-y\right\|)-\varphi(\left\|x_k-z\right\|)}>-\varepsilon$ is similar to previous case, so, we omitted.



In this case also, we get
$$
\|x-y\| \ge\|x-z\|.
$$
So, $E$ is remotal in $X.$
\end{proof}

We introduce the following definitions which will be needed in our next results.

\begin{definition}\label{x2}
Let $X$ be a real normed linear space and $M$ be a non-empty subset of $X$. A sequence $\{x_n\}_{n\in \mathbb{N}} \subset M$ is said to be $\alpha \beta$-statistically  maximizing if there exists $x \in X$ such that $\{\left\|x_n-x\right\|\}_{n \in \mathbb{N}}$ is  $\alpha \beta$-statistically convergent to $\delta(x, M)=\sup \{\|x-y\|: y \in M\}$ as $n \rightarrow \infty$.
\end{definition}

\begin{remark}
Note that, Definition \ref{x2}, improves the Definition 2.18 in \cite{ssm} as $\alpha\beta$-statistical convergence is more general than the statistical convergence.
\end{remark}


\begin{theorem}
Let $X$ be a real normed linear space and $M$ be a non-empty bounded subset of $X$. If $\left\{x_n\right\}_{n \in \mathbb{N}}$ is maximizing in $M$ then $\left\{x_n\right\}_{n \in \mathbb{N}}$ is $\alpha \beta$-statistically maximizing in $M$ for every $[\alpha, \beta]$ pair.
\end{theorem}

\textbf{Proof:}
Let the sequence $\left\{x_n\right\}_{n \in \mathbb{N}}$ be maximizing. So, there exists $x \in X$ such that $\left\|x_n-x\right\| \rightarrow \delta(x, M)$ as $n \rightarrow \infty$, Let $\varepsilon>0$. So the set $A=\{k \in \mathbb{N}$ : $\left.\left|\left\|x_k-x\right\|-\delta(x, M)\right| \geqslant \varepsilon\right\}$ is finite. So, $\left\|x_n-x\right\|$ is $\alpha\beta$-statistically convergent to $\delta(x, M).$ So, $\left\{x_n\right\}_{n \in \mathbb{N}}$ is $\alpha\beta$-statistically maximizing. \qedsymbol

\begin{example}
    
Let $\mathbb{R}$ be the set of all real numbers and $M=[-1,1]$. Now $\delta(0, M)=\sup \{|x|: x \in M\}=1$. Let $\left\{x_n\right\}_{n \in N}$ in $[-1,1]$ be defined as,

$$
x_n=\left\{\begin{array}{cc}
0, & \mbox{if}~n=2^m~\mbox{for some}~m\in \mathbb{N}; \\
1-c^n, & \mbox{if}~  n \neq 2^m~\mbox{for all}~m\in \mathbb{N}
\end{array}\right.
$$
where $0<c<1$. We claim that this sequence is not maximizing but it is $\alpha\beta$-statistically maximizing in $M.$
We choose $\alpha_n=1$ and $\beta_n=n^2$ for all $n\in \mathbb{N}.$
The sequence $\left\{x_n\right\}_{n \in \mathbb{N}}$ is not maximizing as, for any $x \in \mathbb{R}$, the real sequence $\left\{\left|x_n-x\right|\right\}_{n \in \mathbb{N}}$ is not convergent in $\mathbb{R}$. But, $\left\{x_n\right\}_{n \in \mathbb{N} }=\left\{\left|x_n-0\right|\right\}_{n \in \mathbb{N}}$ is $\alpha \beta$-statistically convergent to $\delta(0, M)=1$. Indeed, let $0 \leqslant \varepsilon \leqslant 1$. Now,
$$
\begin{aligned}
& \frac{1}{n^2}\left|\left\{k \in\left[1, n^2\right]:\left|x_k-1\right| \geqslant \varepsilon\right\}\right| \\
&=\frac{1}{n^2}(\left[\log _2 n^2\right]+d) (d~\mbox{is a finite positive integer}) \\
& \leqslant \frac{2 \log _2 n}{n^2}+\frac{d}{n^2}\\
& \Longrightarrow \lim _{n \rightarrow \infty} \frac{1}{n^2}\left|\left\{k \in\left[1, n^2\right]:\left|x_k-1\right| \geqslant \varepsilon\right\}\right|=0.
\end{aligned}
$$
\end{example}
Hence the sequence $\left\{x_n\right\}_{n \in \mathbb{N}}$ is $\alpha \beta$-statistically maximizing in $M$.





In \cite{CBM1}, Sababheh et al. introduced the notion of $x$-compactness with respect to a closed bounded subset $E$ of a Banach space $X$ and $x\in X.$ Firstly we recall the definition from \cite{CBM1} as follows. In upcoming definitions and results $diam$ indicates the diameter of the concerned set.

\begin{definition}
Let $X$ be a Banach space and $E$ be a closed bounded subset of $X.$ Let $x\in X.$ Then $E$ is said to be $x$-compact if $diam(E-B_n(x))\rightarrow 0$ as $n \rightarrow \infty$ where $B_n(x)=\{y \in X : \left\|y-x \right\| \leq \delta(x,E)-\frac{1}{n}\}.$   
\end{definition}

We now introduce a concept called $x$-$\alpha\beta$-compactness in a Banach space $X$ where as usual $(\alpha,\beta)$ is a pair of sequences satisfying the three conditions stated in Definition \ref{y1}. It may happen that there exists a Banach space $X,$ a closed bounded subset $E$ and a point $x\in X$ such that $E$ is not $x$-compact but $E$ is $x$-$\alpha\beta$-compact.

\begin{definition}
Let $X$ be a Banach space and $E$ be a closed bounded subset of $X.$ Let $x\in X.$ Then $E$ is said to be $x$-$\alpha\beta$-compact if there exists a sequence $\{t_n\}_{n \in \mathbb{N}}$ of non-negetive real numbers such that $(diam(E-B_{t_n}(x)))$ is $\alpha\beta$-statistically convergent to $0$ as $n \rightarrow \infty$ where $B_{t_n}(x)=\{y \in X : \left\|y-x \right\| \leq \delta(x,E)-t_n\}$ and $\{t_n\}_{n \in \mathbb{N}}$ is $\alpha\beta$-statistically convergent to $0.$   
\end{definition}

\begin{theorem}\label{z1}
Let $X$ be a Banach space, $E$ be a closed bounded subset of $X$ and $x\in X.$  If $E$ is $x$-compact then $E$ is $x$-$\alpha\beta$-compact.   
\end{theorem}

\begin{proof}
Since $E$ is $x$-compact, so, $diam(E-B_n(x))\rightarrow 0$ as $n \rightarrow \infty$ where $B_n(x)=\{y \in X : \left\|y-x \right\| \leq \delta(x,E)-\frac{1}{n}\}.$ Now, here the sequence is $t_n=\frac{1}{n}, n\in \mathbb{N}$ and for every $\varepsilon>0,$ $A_{\alpha\beta}^n(\varepsilon)=\{k\in [\alpha_n, \beta_n]: t_n \geq \varepsilon\}$ is finite. So, $\{t_n\}_{n \in \mathbb{N}}$ is $\alpha\beta$-statistically convergent to $0.$ Also, since $diam(E-B_n(x))\rightarrow 0$ as $n \rightarrow \infty$, so, $diam(E-B_n(x))$ is $\alpha\beta$-statistically convergent to $0.$ So, $E$ is $x$-$\alpha\beta$-compact.     
\end{proof}

The converse of Theorem \ref{z1} is not true. This follows from the upcoming example.

\begin{example}
Consider the Banach space $\mathbb{R}$ and $M=[-1,1].$ Let $0\in \mathbb{R}.$  Now, $\delta(0,M)=1$ and, $B_n(0)=\{x \in \mathbb{R}: |x|\leq 1-\frac{1}{n}\}.$ In this case $M$ is not $0$-compact because $diam(M-B_n(0))\geq 2$ for all $n\in \mathbb{N}$ and $diam(M-B_n(0))$ is not convergent to $0.$ Let $\alpha_n=1$ and $\beta_n=n^2$ for all $n\in \mathbb{N}.$ Let us define a sequence $(t_n)$ by 

$$
t_n=\left\{\begin{array}{cc}
0 &~\mbox{if}~n \neq m^2~\mbox{for any}~m\in \mathbb{N}; \\
1 &~\mbox{if}~n = m^2~\mbox{for some}~m\in \mathbb{N}.
\end{array}\right.
$$
So, $(t_n)$ is $\alpha\beta$-statistically convergent to $0.$ Now, it can be seen that 

$$
diam(M-B_{t_n}(0)=\left\{\begin{array}{cc}
0 &~\mbox{if}~n \neq m^2~\mbox{for any}~m\in \mathbb{N}; \\
2 &~\mbox{if}~n = m^2~\mbox{for some}~m\in \mathbb{N}.
\end{array}\right.
$$
So, $(diam(M-B_{t_n}(0))$ is $\alpha\beta$-statistically convergent to $0.$ This shows that $M$ is $0$-$\alpha\beta$-compact.
\end{example}

Now, we discuss some results, related to the $\alpha\beta$-compactness which will help us to find sufficient condition for remotality and uniquely remotality of closed bounded subsets in a Banach space $X.$

\begin{theorem}\label{z2}
Let $X$ be a Banach space and $E$ be a closed bounded subset of $X.$ If $E$ is $x$-$\alpha\beta$-compact for some $x\in X$ then $\delta(x,E)=\sup\{\left\|x-z\right\| : z\in E\}$ is attained.
\end{theorem}

\begin{proof}
Since $E$ is $x$-$\alpha\beta$-compact, so, there exists a sequence $(t_n)$ of non-negetive real numbers such that $(diam(E-B_{t_n}(x)))$ is $\alpha\beta$-statistically convergent to $0$ as $n \rightarrow \infty$ where $B_{t_n}(x)=\{y \in X : \left\|y-x \right\| \leq \delta(x,E)-t_n\}$ and $(t_n)$ is $\alpha\beta$-statistically convergent to $0.$ We can choose a sequence $(h_n)\subset E$ such that $\left\|x-h_n\right\|\uparrow \delta(x,E).$ Let $\varepsilon>0.$ Now, we have,
$$\displaystyle{\lim_{n \rightarrow \infty}}\frac{1}{(\beta_n-\alpha_n+1)}|\{k\in [\alpha_n,\beta_n]: diam(E-B_{t_n}(x))\geq \varepsilon\}|=0.$$ This means, the set $A_{\alpha\beta}^n(\varepsilon)=\{k: diam(E-B_{t_n}(x)) \geq \varepsilon\}$ has $\alpha\beta$-density zero. So, the set $\mathbb{N}-A_{\alpha\beta}^n(\varepsilon)$ has $\alpha\beta$-density 1. Let $N\in \mathbb{N}-A_{\alpha\beta}^n(\varepsilon).$ Choose a  natural number $M\geq N$ such that $\left\|x-h_n \right\|>\delta(x,E)-t_N~\mbox{for}~n\geq M.$ Now, for $i,j\geq M,~\mbox{we have}~\left\|h_i-h_j \right\|<\varepsilon.$ Let $h_n \rightarrow h\in E$ as $n\rightarrow \infty.$ This shows that $\left\|x-h \right\|=\delta(x,E).$ So, $\delta(x,E)$ is attained. 
\end{proof}

\begin{corollary}
Let $X$ be a Banach space and $E$ be a closed bounded subset of $X.$ If $E$ is $x$-$\alpha\beta$-compact for every $x\in X$ then $E$ is remotal.
\end{corollary}

\begin{theorem}
Let $X$ be a Banach space and $E$ be a closed bounded subset of $X.$ If $E$ is $x$-$\alpha\beta$-compact for some $x\in X$ then $x$ is a max-Chebyshev point.
\end{theorem}

\begin{proof}
By Theorem \ref{z2}, we can say $\delta(x,E)$ is attained by some element in $E.$ If possible, let there exist $h, h'\in E$ such that $\left\|x-h \right\|=\left\|x-h' \right\|=\delta(x,E).$ Since $E$ is $x$-$\alpha\beta$-compact, so, there exists a sequence $(t_n)$ of non-negetive real numbers such that $(diam(E-B_{t_n}(x)))$ is $\alpha\beta$-statistically convergent to $0$ as $n \rightarrow \infty$ where $B_{t_n}(x)=\{y \in X : \left\|y-x \right\| \leq \delta(x,E)-t_n\}$ and $(t_n)$ is $\alpha\beta$-statistically convergent to $0.$ So, $\left\|x-h \right\|=\left\|x-h' \right\|=\delta(x,E)>\delta(x,E)-t_n~\mbox{for all}~n\in \mathbb{N}.$ So, $h,h' \in E-B_{t_n}(x)~\mbox{for all}~n\in \mathbb{N}.$ Also, $\left\|h-h' \right\|\leq diam(E-B_{t_n}(x)).$ Since, $(diam(E-B_{t_n}(x)))$ is $\alpha\beta$-statistically convergent to $0,$ so, $\left\|h-h' \right\|$ is $\alpha\beta$-statistically convergent to $0.$ So, $\left\|h-h' \right\|=0$ and $h=h'.$  
\end{proof}

\begin{corollary}
Let $X$ be a Banach space and $E$ be a closed bounded subset of $X.$ If $E$ is $x$-$\alpha\beta$-compact for every $x\in X$ then $E$ is uniquely remotal.
\end{corollary}

Now, in the last part we will introduce the notion of partial $x$-$\alpha\beta$-compactness and give one result related to this.

\begin{definition}
Let $X$ be a Banach space, $E$ be a closed bounded subset of $X$ and $x\in X.$ Then $E$ is said to be partial $x$-$\alpha\beta$-compact if there exists closed $H\subset E$ such that $\delta(x,E)=\delta(x,H)$ and $H$ is $x$-$\alpha\beta$-compact.  
\end{definition}

\begin{theorem}
Let $X$ be a Banach space and $E$ be a closed bounded subset of $X.$ Then $E$ is partial $x$-$\alpha\beta$-compact if and only if $\delta(x,E)$ is attained.     
\end{theorem}

\begin{proof}
First of all let $E$ be partial $x$-$\alpha\beta$-compact. So, there exists $H\subset E$ such that $\delta(x,E)=\delta(x,H)$ and $H$ is $x$-$\alpha\beta$-compact. So, by Theorem \ref{z2}, there exists $e\in H$ such that $\delta(x,H)=\left\|x-e \right\|.$ This shows that $\delta(x,E)$ is attained. Now, let $E$ be such that $\delta(x,E)$ is attained. So, there exists $e_1\in E$ such that $\delta(x,E)=\left\|x-e_1 \right\|.$ We can take $H=\{e_1\}$ and $H$ is $x$-$\alpha\beta$-compact.            
\end{proof}

\section{Conflict of Interest}
The authors of the article declares that they have no conflict of interest.

\end{document}